\documentclass[a4paper,10pt,leqno]{amsart}
\usepackage{hyperref}
\usepackage{enumerate,amssymb}
\usepackage[arrow,curve,matrix,tips,2cell]{xy}
  \SelectTips{eu}{10} \UseTips
  \UseAllTwocells

\DeclareMathAlphabet{\matheurm}{U}{eur}{m}{n}

\DeclareMathOperator{\ANR}{ANR}

\DeclareMathOperator{\id}{id}

\DeclareMathOperator{\inti}{int}

\DeclareMathOperator{\PL}{PL}

\DeclareMathOperator{\Wh}{Wh}

  \newcommand{\IH}{\mathbb{H}}

  \newcommand{\IR}{\mathbb{R}}
  \newcommand{\IS}{\mathbb{S}}

  \newcommand{\IZ}{\mathbb{Z}}

  \newcommand{\calh}{\mathcal{H}}

  \newcommand{\calo}{\mathcal{O}}

  \newcommand{\calu}{\mathcal{U}}

  \newcommand{\bfK}{{\mathbf K}}
  \newcommand{\bfL}{{\mathbf L}}

\newcommand{\xycomsquare}[8]   
{\xymatrix{#1 \ar[r]^{#2} \ar[d]^{#4} &
#3 \ar[d]^{#5}  \\
#6\ar[r]^{#7} &
#8
}
}

\newcommand{\xycomsquareminus}[8] 
{\xymatrix{#1 \ar[r]^-{#2} \ar[d]^-{#4} &
#3 \ar[d]^-{#5}  \\
#6\ar[r]^-{#7} &
#8
}
}

\newcounter{commentcounter}

\theoremstyle{plain}
\newtheorem{theorem}{Theorem}[section]

\newtheorem{lemma}[theorem]{Lemma}
\newtheorem{corollary}[theorem]{Corollary}
\newtheorem{proposition}[theorem]{Proposition}

\newtheorem*{theorem*}{Theorem}
\newtheorem*{theoremA*}{Theorem A}
\newtheorem*{theoremB*}{Theorem B}

\theoremstyle{definition}
\newtheorem{definition}[theorem]{Definition}

\newtheorem{example}[theorem]{Example}

\theoremstyle{remark}

\makeatletter\let\c@equation=\c@theorem\makeatother

\newcommand{\x}{{\times}}

\newcommand{\dd}{{\partial}}
\newcommand{\e}{{\varepsilon}}

\title{On hyperbolic groups with spheres as boundary}
           \author{Arthur Bartels}
            \email{bartelsa@math.uni-muenster.de}
          \urladdr{http://www.math.uni-muenster.de/u/bartelsa}
          \address{Westf\"alische Wilhelms-Universit\"at M\"unster\\
                   Mathematisches Institut\\
                   Einsteinstr.~62,
                   D-48149 M\"unster, Germany}
           \author{Wolfgang L\"uck}
            \email{lueck@math.uni-muenster.de}
          \urladdr{http://www.math.uni-muenster.de/u/lueck}
          \address{Westf\"alische Wilhelms-Universit\"at M\"unster\\
                   Mathematisches Institut\\
                   Einsteinstr.~62,
                   D-48149 M\"unster, Germany}
           \author{Shmuel Weinberger}
            \email{shmuel@math.uchicago.edu}
          \urladdr{http://www.math.uchicago.edu/\%7Eshmuel/}
          \address{Department of Mathematics\\
                   University of Chicago\\
                   5734 S. University Avenue
                   Chicago, IL 60637-151, U.S.A.}
             \date{November 2009}
       \dedicatory{Dedicated to Steve Ferry on the occasion of his 
                     60th birthday}

\begin{document}

\begin{abstract}
Let $G$ be a torsion-free hyperbolic group and let $n \ge 6$ be an
integer. 
We prove that $G$ is the fundamental group of
a closed aspherical manifold if the boundary of 
$G$ is homeomorphic to an $(n-1)$-dimensional  sphere.  
\end{abstract}

\maketitle


\typeout{------- Introduction ------}

\section*{Introduction}

If $G$ is the fundamental group of an $n$-dimensional 
closed Riemannian manifold with negative sectional curvature, 
then $G$ is a hyperbolic group in the sense of Gromov  (see for
instance~\cite{Bowditch(1991)}, 
\cite{Bridson-Haefliger(1999)},
\cite{Ghys-Harpe(1990)}, \cite{Gromov(1987)}).
Moreover such a group is torsion-free and its
boundary $\partial G$ is homeomorphic to
a sphere.
This leads to the natural question
whether a torsion-free hyperbolic group with a 
sphere as boundary occurs as
fundamental group of a closed aspherical manifold
(see Gromov~\cite[page~192]{Gromov(1993)}). 
We settle this question if the dimension of the sphere is
at least $5$.

\begin{theoremA*}
Let $G$ be a torsion-free hyperbolic group and let $n$ be an 
integer  $\geq 6$. The following statements are equivalent:
\begin{enumerate}
  \item \label{the:main:sphere}
        the boundary $\dd G$ is homeomorphic to $S^{n-1}$;
  \item \label{the:main:manifold}
        there is a closed aspherical topological manifold $M$ 
        such that $G \cong \pi_1(M)$, 
        its universal covering $\widetilde{M}$ is 
        homeomorphic to $\IR^n$
        and the compactification of $\widetilde{M}$ by 
        $\partial G$ is homeomorphic to $D^n$;
\end{enumerate}
\end{theoremA*}  

The aspherical manifold $M$ appearing in our result
is unique up to homeomorphism.
This is a consequence of the validity of the Borel Conjecture
for hyperbolic groups~\cite{Bartels-Lueck(2009borelhyp)},
see also Section~\ref{sec:rigidity}.   

The proof depends on the surgery theory for
homology $\ANR$-manifolds due to Bryant-Ferry-Mio-Weinberger~%
\cite{Bryant-Ferry-Mio-Weinberger(1996)}
and the validity of the $K$- and $L$-theoretic Farrell-Jones Conjecture for
hyperbolic groups due to
Bartels-Reich-L\"uck~\cite{Bartels-Lueck-Reich(2008hyper)}
and Bartels-L\"uck~\cite{Bartels-Lueck(2009borelhyp)}. 
It seems likely that this result holds also if $n = 5$.
Our methods can be extended to this case 
if the surgery theory from~\cite{Bryant-Ferry-Mio-Weinberger(1996)} 
can be extended to the case of $5$-dimensional
homology $\ANR$-manifolds -- such an extension has been announced by
Ferry-Johnston.  We also hope to give a treatment elsewhere by more
algebraic methods.

We do not get information in dimensions $n \le 4$
for the usual problems about surgery. 
For instance, our methods
give no information in the case, where the boundary is homeomorphic to 
$S^3$, since virtually cyclic groups are the only hyperbolic groups
which are known to be good in the sense of 
Friedman~\cite{Freedman(1983)}. 
In the case $n = 3$ 
there is the conjecture of Cannon~\cite{Cannon(1991)} that
a group $G$ acts properly, isometrically and cocompactly on the
$3$-dimensional hyperbolic plane $\IH^3$ if and only if it is a
hyperbolic group whose boundary is homeomorphic to $S^2$. 
Provided that the infinite hyperbolic group $G$ occurs as the 
fundamental group of a closed irreducible $3$-manifold,
Bestvina-Mess~\cite[Theorem~4.1]{Bestvina-Mess(1991)} 
have shown that its universal cover is homeomorphic to 
$\IR^3$ and its compactification by $\partial G$ is 
homeomorphic to $D^3$, and the Geometrization Conjecture of
Thurston implies that $M$ is hyperbolic and 
$G$ satisfies Cannon's conjecture.
The problem is solved in the case $n = 2$, 
essentially as a consequence of Eckmann's theorem that 2
dimensional Poincare duality groups are surface groups
(see \cite{Eckmann(1987)}).
Namely, 
for a hyperbolic group $G$
its boundary $\partial G$ is homeomorphic to $S^1$ if and only if
$G$ is a Fuchsian group 
(see~\cite{Casson-Jungreis(1994)}, \cite{Freden(1995)},
\cite{Gabai(1991)}). 

In general the boundary of a hyperbolic group is not locally a 
Euclidean space but has a fractal behavior. 
If the boundary $\partial G$ of an infinite
hyperbolic group $G$ contains an open subset homeomorphic to 
Euclidean $n$-space, then it is homeomorphic to $S^n$.
This is proved in~\cite[Theorem~4.4]{Kapovich+Benakli(2002)},
where more information about the boundaries of hyperbolic groups can
be found.

We also prove the following result.

\begin{theoremB*}
  Let $G$ and $H$ be a torsion-free hyperbolic groups such that
  $\dd G \cong \dd H$.
  Then $G$ can be realized as the fundamental group of a closed 
  aspherical manifold of dimension at least $6$ if and only if 
  $H$ can be realized as the fundamental group of such a manifold.

  Moreover, even in case that
  neither can be realized by a closed aspherical manifold, they can both
  be realized by closed  aspherical homology $\ANR$-manifolds, which both
  have the same Quinn obstruction~\cite{Quinn(1987_resolution)} 
  (see Theorem~\ref{thm:Quinn-obstruction} for a review of
  this notion).
\end{theoremB*}

In particular, if $G$ is hyperbolic and realized as the fundamental group
of a closed aspherical manifold of dimension at least $6$, then any
torsion-free group $H$ that is quasi-isometric to $G$ can also
be realized as the fundamental group of such a manifold.
This follows from Theorem~B, because the homeomorphism type of the 
boundary of a hyperbolic group
is invariant under quasi-isometry (and so is the property of being
hyperbolic).
The attentative reader will realize that most of the content 
of Theorem~A can also be deduced from Theorem~B,
as every sphere appears as the boundary of the fundamental group
of some closed hyperbolic manifold. 

This paper was financially supported by the
Sonderforschungsbereich 478 \--- Geometrische Strukturen in der
Mathematik \---, the Max-Planck-Forschungspreis 
and the Leibniz-Preis of the second author and
NSF grant 0852227 of the third author. 

The techniques and ideas of this paper are 
very closely related to the work of Steve Ferry;
indeed his unpublished work could have been used to simplify some
parts of this work.
It is a pleasure to dedicate this paper to him on the occasion 
of his 60th birthday.


\typeout{------- Homology manifolds ------}

\section{Homology manifolds}

A topological space $X$ is called an 
\emph{absolute neighborhood retract} 
or briefly an $\ANR$ if it is normal and for every normal space $Z$, 
every closed subset $Y \subseteq Z$ and 
every (continuous) map $f \colon Y \to X$ there exists 
an open neighborhood $U$ of $Y$ in $Z$ together with an 
extension $F \colon U \to X$ of $f$ to $U$.  

\begin{definition} [Homology $\ANR$-manifold]
  \label{def:homology-ANR-manifold}
  An \emph{$n$-dimensional homology $\ANR$-manifold $X$} is an 
  absolute neighborhood retract satisfying:
  \begin{itemize}
  \item $X$ has a countable base for its topology;
  \item the topological dimension of $X$ is finite;
  \item $X$ is locally compact;
  \item for every $x \in X$ the $i$-th singular homology group 
        $H_i(X,X-\{x\})$ is trivial
        for $i \not= n$ and infinite cyclic for $i = n$. 
  \end{itemize}
\end{definition}

Notice that a normal space with a countable basis for its topology
is metrizable by the Urysohn Metrization Theorem 
(see~\cite[Theorem~4.1 in Chapter~4-4 on page~217]{Munkres(1975)})
and is separable, i.e., contains a countable dense 
subset~\cite[Theorem~4.1]{Munkres(1975)}. 
Notice furthermore that every metric space is normal
(see~\cite[Theorem~2.3 in Chapter~4-4 on page~198]{Munkres(1975)}),
and has a countable basis for its topology 
if and only if it is separable
(see~\cite[Theorem~1.3 in Chapter~4-1 on page~191 and 
         Exercise~7 in Chapter~4-1 on page~194]{Munkres(1975)}).
Hence a homology $\ANR$-manifold in the sense of 
Definition~\ref{def:homology-ANR-manifold}
is the same as a generalized manifold in the sense of
Daverman~\cite[page~191]{Daverman(1986)}. 
A closed $n$-dimensional topological manifold is an 
example of a closed $n$-dimensional homology $\ANR$-manifold
(see~\cite[Corollary 1A in V.26 page~191]{Daverman(1986)}). 
A homology $\ANR$-manifold $M$ is said to have the 
\emph{disjoint disk property (DDP)}, if for any $\e > 0$
and maps $f, g \colon D^2 \to M$, there are maps
$f', g' \colon D^2 \to M$ so that $f'$ is $\e$-close to $f$,
$g'$ is $\e$-close to $g$ and $f'(D^2) \cap g'(D^2) = \emptyset$,
see for example~\cite[page~435]{Bryant-Ferry-Mio-Weinberger(1996)}.
We recall that a \emph{Poincar\'e duality group} $G$ is a 
finitely presented group satisfying 
the following two conditions:
firstly, the $\IZ G$-module $\IZ$ (with the trivial $G$-action) admits
a resolution of finite length by finitely generated 
projective $\IZ G$-modules;
secondly, there is $n$ such that $H^i(G;\IZ G) = 0$ for $n \neq i$ and
$H^n(G;\IZ G) \cong \IZ$. 
In this case $n$ is the formal dimension of the 
Poincar\'e duality group $G$.

\begin{theorem}
  \label{thm:FJ_and_Borel-existence}
  Let $G$ be a torsion-free group.
  \begin{enumerate}

  \item \label{thm:FJ_and_Borel-existence:ex} Assume that
  \begin{itemize}
  \item the (non-connective) $K$-theory assembly map
        \[
         H_i\bigl(BG;\bfK_\IZ\bigr) \to  
            K_i(\IZ G)
        \]
        is an isomorphism for $i \leq 0$ and
        surjective for $i = 1$;
  \item the (non-connective) $L$-theory assembly map 
        \[
         H_i\bigl(BG;^w\bfL^{\langle -\infty\rangle}_\IZ\bigr) \to  
            L_i^{\langle - \infty \rangle}(\IZ G,w)
        \]
        is bijective for every $i \in \IZ$ and every 
        orientation homomorphism $w \colon G \to \{\pm 1\}$.
  \end{itemize}
  Then for $n \geq 6$ the following are equivalent:
  \begin{enumerate}
    \item \label{thm:FJ_and_Borel:duality}
          $G$ is a Poincar\'e duality group of formal dimension $n$;
    \item \label{thm:FJ_and_Borel:homology-mfd}
          there exists a closed $\ANR$-homology manifold
          $M$ homotopy equivalent to $BG$. 
          In particular, $M$ is aspherical and $\pi_1(M) \cong G$;
  \end{enumerate}
  
  \item \label{thm:FJ_and_Borel-existence:DDP} 
  If the statements in assertion~\ref{thm:FJ_and_Borel-existence:ex}
  hold, then the homology 
  $\ANR$-manifold $M$ appearing there
  can be arranged to have the DDP;

  \item \label{thm:FJ_and_Borel-existence:uniqueness}
  If the statements in assertion~\ref{thm:FJ_and_Borel-existence:ex}
  hold, then the homology 
  $\ANR$-manifold $M$ appearing there is
  unique up to $s$-cobordism of $\ANR$-homology manifolds.
\end{enumerate}
\end{theorem}

\begin{proof}~\ref{thm:FJ_and_Borel-existence:ex}
  The assumption on the $K$-theory assembly map
  implies that $\Wh(G) = 0$, 
  $\tilde{K}_0(\IZ G) = 0$
  and $K_i(\IZ G) = 0$ for $i < 0$,
  compare~\cite[Conjecture~1.3~on~page~653 and Remark~2.5~on~page~679]
  {Lueck-Reich(2005)}.
  This implies that we can change the decoration in the
  above $L$-theory assembly map from $\langle -\infty \rangle$ to $s$
  (see~\cite[Proposition~1.5~on~page~664]{Lueck-Reich(2005)}).
  Thus the assembly map $A$ in the algebraic surgery exact
  sequence~\cite[Definition~14.6]{Ranicki(1992)} 
  (for $R = \IZ$ and $K = BG$) is an isomorphism.
  This implies in particular that the 
  quadratic structure groups $\IS_i(\IZ,BG)$
  are trivial for all $i \in \IZ$.
 
  Assume now that $G$ is a Poincar\'e duality group of dimension $n \ge 3$. 
  We conclude from
  Johnson-Wall~\cite[Theorem~1]{Johnson+Wall(1972)} that
  $BG$ is a finitely dominated $n$-dimensional Poincar\'e complex in the 
  sense of  Wall~\cite{Wall(1967)}.
  Because $\tilde K_0(\IZ G) = 0$ the finiteness obstruction
  vanishes and hence $BG$ can be realized as a finite $n$-dimensional 
  simplicial complex
  (see~\cite[Theorem~F]{Wall(1965a)}).
  We will now use Ranicki's ($4$-periodic) total surgery 
  obstruction $\overline{s}(BG) \in \overline{\IS}_n(BG)$ of
  the Poincar\'e complex $BG$,
  see~\cite[Definition~25.6]{Ranicki(1992)}. 
  The main result of \cite{Bryant-Ferry-Mio-Weinberger(1996)}
  asserts that this obstruction vanishes if and only if there
  is a closed $n$-dimensional homology $\ANR$-manifold $M$ homotopy
  equivalent to $BG$.
  The groups $\overline{\IS}_k(BG)$ arise in a $0$-connected 
  version of the algebraic surgery 
  sequence~\cite[Definition~15.10]{Ranicki(1992)}.
  It is a consequence 
  of~\cite[Proposition~15.11(iii)]{Ranicki(1992)}
  (and the fact that $L_{-1}(\IZ) = 0$)
  that $\overline{\IS}_n(BG) = \IS_n(\IZ,BG)$.
  Since $\IS_n(\IZ,BG) = 0$, we conclude $\overline{s}(BG) = 0$.
  This shows that~\ref{thm:FJ_and_Borel:duality}
  implies~\ref{thm:FJ_and_Borel:homology-mfd}.
  (In this argument we ignored that the orientation
  homomorphism $w \colon G \to \{\pm 1\}$ may be non-trivial.
  The argument however extends to this case, 
  compare~\cite[Appendix~A]{Ranicki(1992)}.)
  Homology manifolds satisfy Poincar\'e duality and
  therefore~\ref{thm:FJ_and_Borel:homology-mfd} 
  implies~\ref{thm:FJ_and_Borel:duality}. 
  \\[1mm]~\ref{thm:FJ_and_Borel-existence:DDP}
  It is explained 
  in~\cite[Section~8]{Bryant-Ferry-Mio-Weinberger(1996)}  
  that this homology manifold $M$ appearing above can be arranged to 
  have the DDP.
  (Alternatively, we could appeal 
  to~\cite{Bryant-Ferry-Mio-Weinberger(2007)} and resolve $M$
  by an $n$-dimensional homology $\ANR$-manifold with the DDP.)
  \\[1mm]~\ref{thm:FJ_and_Borel-existence:uniqueness}
  The uniqueness statement follows from
  Theorem~\ref{thm:rigidity}~\ref{the:uniqueness_ANR}.
\end{proof}

In order to replace homology $\ANR$-manifolds by topological
manifolds we will later use the following result that combines
work of Edwards and Quinn,  
see~\cite[Theorems~3 and~4 on page~288]{Daverman(1986)}, 
\cite{Quinn(1987_resolution)}).

\begin{theorem}
  \label{thm:Quinn-obstruction}
  There is an invariant $\iota(M) \in 1 + 8 \IZ$ 
  (known as the Quinn obstruction)
  for homology
  $\ANR$-manifolds with the following properties:
  \begin{enumerate}
  \item \label{thm:Quinn-obst:local}
        if $U \subset M$ is an open subset, then 
        $\iota(U) = \iota(M)$;
  \item \label{thm:Quinn-obst:manifold}
        let $M$ be a homology $\ANR$-manifold of dimension $\geq 5$.
        Then the following are equivalent
        \begin{itemize}
        \item $M$ has the DDP and $\iota(M) = 1$;
        \item $M$ is a topological manifold.
        \end{itemize}        
  \end{enumerate}
\end{theorem}

\begin{definition}
  \label{def:homology-mfd-with-boundary}
  An \emph{$n$-dimensional homology $\ANR$-manifold $M$ with boundary
  $\dd M$} is an absolute neighborhood retract which is a
  disjoint union $M = \inti M \cup \dd M$, where
  \begin{itemize}
  \item $\inti M$ is an $n$-dimensional homology $\ANR$-manifold;
  \item $\dd M$ is an $(n-1)$-dimensional homology $\ANR$-manifold;
  \item for every $z \in \dd M$ the singular homology group
        $H_i(M, M \setminus \{ z \})$ vanishes for all $i$.
  \end{itemize}
\end{definition}

\begin{lemma}
  \label{lem:collar-for-homology-mfd-boundary}
  If $M$ is an $n$-dimensional homology $\ANR$-manifold with
  boundary, then $\widehat{M} :=  M \cup_{\dd M} \dd M \x [0,1)$
  is an $n$-dimensional homology $\ANR$-manifold. 
\end{lemma}

\begin{proof}
  Suppose that $Y$ is the union of two closed subsets 
  $Y_1$ and $Y_2$ and set $Y_0 := Y_1 \cap Y_2$. 
  If $Y_0$, $Y_1$ and $Y_2$ are $\ANR$s, then $Y$
  is an $\ANR$, see~\cite[Theorem~7 on page~117]{Daverman(1986)}.
  If $Y_1$ and $Y_2$ have countable bases $\calu_1$ and
  $\calu_2$ of the topology, then sets
  $U_1 \setminus Y_2$ with $U_1 \in \calu_1$,
  $U_2 \setminus Y_1$ with $U_2 \in \calu_2$
  and $(U_1 \cup U_2)^\circ$ with $U_i \in \calu_i$
  form a countable basis of the topology of $Y$.
  (Here $( \; )^\circ$ is the operation of taking the
  interior in $Y$.)
  If $Y_1$ and $Y_2$ are both finite dimensional,
  then $Y$ is finite 
  dimensional~\cite[Theorem~9.2 on page~303]{Munkres(1975)}.
  If $Y_1$ and $Y_2$ are both locally compact, then 
  $Y$ is locally compact.

  Thus the only non-trivial requirement 
  is that for 
  $x = (z,0) \in \widehat{M}$
  with $z \in \dd M$,
  we have $H_i(\widehat{M}, \widehat{M} \setminus \{ x \}) = 0$
  if $i \neq n$ and $\cong \IZ$ if $i=n$.
  Let $I_z := \{ z \} \x [0,1/2)$.
  Because of homotopy invariance we can replace $\{ x \}$ by
  $I_z$.
  Let $U_1 := M \cup_{\dd M} \dd M \x [0,1/2) \subset \widehat{M}$ and
  $U_2 :=  \dd M \x (0,1) \subset \widehat{M}$.
  Then
  $H_i(U_1,U_1 \setminus I_z) \cong H_i(M,M \setminus \{ z \}) = 0$
  and $H_i(U_2,U_2 \setminus I_z) = 0$.
  Because $U_1$ and $U_2$ are both open, we can use a
  Mayer-Vietoris sequence to deduce
  \[
  H_i( \widehat{M}, \widehat{M} \setminus I_z ) \cong
  H_{i-1}( U_1 \cap U_2, U_1 \cap U_2 \setminus I_z) \cong
  H_{i-1}(\dd M, \dd M \setminus \{ z \} ).
  \]
  The result follows as $\dd M$ is an $(n-1)$-dimensional
  homology $\ANR$-manifold.
\end{proof}

\begin{corollary}
  \label{cor:quinn-obstr_and_boundary}
  Let $M$ be an homology $\ANR$-manifold with boundary $\dd M$.
  If $\dd M$ is a manifold, then $\iota(\inti M) = 1$.  
\end{corollary}

\begin{proof}
  We use $\widehat{M}$ from 
  Lemma~\ref{lem:collar-for-homology-mfd-boundary}.
  If $\dd M$ is a manifold then so is $\dd M \x (0,1)$.
  The result follows now from Theorem~\ref{thm:Quinn-obstruction}.
\end{proof}


\typeout{------   Hyperbolic groups and aspherical manifolds ------}
                 
\section{Hyperbolic groups and aspherical manifolds}
\label{sec:hyperbolic-groups}

For a hyperbolic group we write 
$\overline{G} := G \cup \dd G$ for the
compactification of $G$ by its boundary,
compare~\cite[III.H.3.12]{Bridson-Haefliger(1999)},
\cite{Bestvina-Mess(1991)}. 
Left multiplication of $G$ on $G$ extends to a natural
action of $G$ on $\overline{G}$.
We will use the following properties of the topology 
on $\overline{G}$.

\begin{proposition}
  \label{prop:topology-of-overline-G}
  Let $G$ be a hyperbolic group.
  Then
  \begin{enumerate}
  \item \label{prop:overline-G:compact}
        $\overline{G}$ is compact;
  \item \label{prop:overline-G:dim}
        $\overline{G}$ is finite dimensional;
  \item \label{prop:overline-G:empty-int}
        $\dd G$ has empty interior in $\overline{G}$;
  \item \label{prop:overline-G:small}
        the action of $G$ on $\overline{G}$ is small at infinity:
        if $z \in \dd G$, $K \subset G$ is finite and 
        $U \subset \overline{G}$ is a neighborhood of $z$,
        then there exists a neighborhood $V \subseteq \overline{G}$ of
        $z$ with $V \subseteq U$ such that for any 
        $g \in G$ with $gK \cap V \neq \emptyset$
        we have $gK \subseteq U$;
  \item \label{prop:overline-G:also-small}
        if $z \in \dd G$ and $U$ is an open neighborhood
        of $z$ in $\overline{G}$, then for every finite subset 
        $K \subseteq G$ there is an open neighborhood $V$
        of $z$ in $\overline{G}$ such that  $V \subseteq U$ and
        $(V \cap G) \cdot K \subseteq U \cap G$.
  \end{enumerate}
\end{proposition}

\begin{proof}~\ref{prop:overline-G:compact} see for 
  instance~\cite[III.H.3.7(4)]{Bridson-Haefliger(1999)}.
  \\~\ref{prop:overline-G:dim} see for 
  instance~\cite[9.3.(ii)]{Bartels-Lueck-Reich(2008cover)}.
  \\~\ref{prop:overline-G:empty-int} is obvious from the definition of the
  topology in~\cite{Bestvina-Mess(1991)}. 
  \\~\ref{prop:overline-G:small}
  see for instance~\cite[page~531]{Rosenthal-Schuetz(2005)}.
  \\~\ref{prop:overline-G:also-small} follows
  from~\ref{prop:overline-G:small}: We may assume $1_G \in K$.
  Pick $V$ as in~\ref{prop:overline-G:small}.
  If $g \in V \cap G$ and $k \in K$,
  then $g \in gK \cap V$. 
  Thus $gK \subseteq U$.
  Therefore $gK \in U \cap G$. 
\end{proof}

Let $X$ be a locally compact space with a cocompact and proper action
of a hyperbolic group $G$.  
Then we equip $\overline{X} := X \cup \dd G$ with the topology
$\calo_{\overline{X}}$
for which a typical open neighborhood of $x \in X$ is an open
subset of $X$ and a typical (not necessarily open) neighborhood
of $z \in \dd G$ is of the form
\[
(U \cap \dd G) \cup (U \cap G) \cdot K 
\]
where $U$ is an open neighborhood of $z$ in $\overline{G}$
and $K$ is a compact subset of $X$ such that $G \cdot K = X$.
We observe that we could fix the choice of $K$ in the 
definition of $\calo_{\overline{X}}$: let $U$, $z$ and $K$ be 
as above and let $K'$ be a further compact subset of $X$ such that
$G \cdot K' = X$.
Because the $G$-action is proper, there is a finite subset $L$
of $G$ such that $K' \subseteq L \cdot K$.
By Proposition~\ref{prop:topology-of-overline-G}~%
\ref{prop:overline-G:also-small} there is an open 
neighborhood $V \subseteq U$ of 
$z \in \overline{G}$ such that
$(V \cap G) \cdot L \subseteq  U \cap G$.
Thus
\[
(V \cap \dd G) \cup (V \cap G) \cdot K'
\subseteq
(U \cap \dd G) \cup (V \cap G) \cdot L \cdot K
\subseteq
(U \cap \dd G) \cup (U \cap G)  \cdot K.
\]

If $f \colon X \to Y$ is a $G$-equivariant continuous map
where $Y$ is also a locally compact space with
a cocompact proper $G$-action, then we define 
$\overline{f} \colon \overline{X} \to \overline{Y}$
by $\overline{f}|_X := f$ and 
$\overline{f}|_{\dd G} := \id_{\dd G}$. 

\begin{lemma}
  Let $G$ be a hyperbolic group and $X$ be a locally compact space 
  with a cocompact and proper $G$-action.  
  \label{lem:overline-X}
  \begin{enumerate}
  \item \label{lem:over-X:compact}
        $\overline{X}$ is compact;
  \item \label{lem:over-X:empty-int}
        $\dd G$ is closed in $\overline{X}$ and its
        interior in $\overline{X}$ is empty;
  \item \label{lem:over-X:dim}
        if  $\dim X$ is finite, then $\dim \overline{X}$
        is also finite;
  \item \label{lem:over-X:funct}
        if $f \colon X \to Y$ is a $G$-equivariant continuous map
        where $Y$ is also a locally compact space with
        a cocompact proper $G$-action, then
        $\overline{f}$ is continuous. 
  \end{enumerate}
\end{lemma}

\begin{proof}
These claims are easily deduced from the observation following the
definition of the topology $\calo_{\overline{X}}$ and 
Proposition~\ref{prop:topology-of-overline-G}.
\end{proof}

We recall that for a hyperbolic group $G$ equipped with
a (left invariant) 
word-metric $d_G$ and a number $d > 0$ the Rips complex
$P_d(G)$ is the simplicial complex whose vertices are the
elements of $G$, and a collection $g_1,\dots,g_k \in G$
spans a simplex if $d_G(g_i,g_j) \leq d$ for all $i,j$.
The action of $G$ on itself by left translation
induces an action of $G$ on $P_d(G)$.
Recall that a closed subset $Z$ in a compact $\ANR$ $Y$ is a \emph{$Z$-set}
if for every open set $U$ in $Y$ the inclusion
$U \setminus Z \to U$ is a homotopy equivalence.
An important result of Bestvina-Mess~\cite{Bestvina-Mess(1991)}
asserts that  (for sufficiently large $d$) $\overline{P_d(G)}$ is an
$\ANR$ such that $\dd G \subset \overline{P_d(G)}$ is $Z$-set.
The proof uses the following criterion 
\cite[Proposition~2.1]{Bestvina-Mess(1991)}:

\begin{proposition}
  \label{prop:criterion-for-Z-set}
  Let $Z$ be a closed subspace of the compact space $Y$
  such that
  \begin{enumerate}
  \item \label{prop:crit:empty-int}
        the interior of $Z$ in $Y$ is empty;
  \item \label{prop:crit:dim}
        $\dim Y < \infty$;
  \item \label{prop:crit:loc-contr}
        for every $k = 0,\dots,\dim Y$, every $z \in Z$ 
        and every neighborhood $U$ of $z$, there is a
        neighborhood $V$ of $z$ such that every map
        $\alpha \colon S^k \to V \setminus Z$
        extends to $\tilde{\alpha} \colon D^{k+1} \to U \setminus Z$;
  \item \label{prop:crit:ANR}
        $Y \setminus Z$ is an $\ANR$. 
  \end{enumerate}
  Then $Y$ is an $\ANR$ and $Z \subset Y$ is a $Z$-set.
\end{proposition}

Condition~\ref{prop:crit:loc-contr}  
is sometimes abbreviated by saying
that $Z$ is $k$-LCC in $Y$, where $k = \dim Y$.

\begin{theorem}
  \label{thm:Z-set-in-X}
  Let $X$ be a locally compact $\ANR$ with a cocompact and proper
  action of a hyperbolic group $G$.
  Assume that there is a $G$-equivariant homotopy 
  equivalence $X \to P_d(G)$.
  If $d$ is sufficiently large, then 
  $\overline{X}$ is an $\ANR$, $\dd G$ is $Z$-set in
  $\overline{X}$ and $Z$ is $k$-LCC in $X$ for all $k$.
\end{theorem}

\begin{proof}
  Bestvina-Mess~\cite[page~473]{Bestvina-Mess(1991)} 
  show that (for sufficiently large $d$)
  $\overline{P_d(G)}$ satisfies the assumptions of
  Proposition~\ref{prop:criterion-for-Z-set}.
  Moreover, they show that $Z$ is $k$-LCC in $\overline{X}$
  for all $k$. 
  Using this, it is not hard to show, that $\overline{X}$ 
  satisfies these assumptions as well:
  Assumptions~\ref{prop:crit:empty-int} and~\ref{prop:crit:dim}
  hold because of Lemma~\ref{lem:overline-X}. 
  Assumption~\ref{prop:crit:ANR} holds because $X$ is an $\ANR$.
  Because $f \mapsto \overline{f}$ is clearly functorial,
  the homotopy equivalence $X \to P_d(G)$ induces a
  homotopy equivalence $\overline{X} \to \overline{P_d(G)}$
  that fixes $\dd G$.
  Using this homotopy equivalence it is easy to check 
  that $\dd G$ is $k$-LCC in $\overline{X}$, 
  because it is $k$-LCC in $\overline{P_d(G)}$.
  Thus Assumption~\ref{prop:crit:loc-contr}
  holds.
\end{proof}

\begin{proposition}
  \label{prop:Z-set-yields-homology-mfd-with-boundary}
  Let $M$ be a finite dimensional locally compact $\ANR$ 
  which is the disjoint union of an $n$-dimensional 
  $\ANR$-homology manifold $\inti M$ and an $(n-1)$-dimensional 
  $\ANR$-homology manifold $\dd M$ such that $\dd M$ is a 
  $Z$-set in $M$.
  Then $M$ is an $\ANR$-homology manifold with boundary $\dd M$.
\end{proposition}

\begin{proof}
  The $Z$-set condition implies that there exists a homotopy
  $H_t \colon M \to M$, $t \in [0,1]$ such that $H_0 = \id_M$ and
  $H_t (M) \subseteq \inti M$ for all $t > 0$, 
  see~\cite[page~470]{Bestvina-Mess(1991)}.
  
  Let $z \in \dd M$. 
  Then the restriction of $H_1$ to $M \setminus \{ z \}$
  is a homotopy inverse for the inclusion
  $M \setminus \{ z \} \to M$.
  Thus $H_i(M, M \setminus \{ z \}) = 0$ for all $i$.
\end{proof}

There is the following (harder) manifold version of
Proposition~\ref{prop:Z-set-yields-homology-mfd-with-boundary}
due to Ferry and Seebeck~\cite[Theorem~5 on page~579]{Ferry(1979)}.

\begin{theorem}
  \label{thm:ferry-seebeck}
  Let $M$ be a locally compact with a countable basis
  of the topology.
  Assume that $M$ is the disjoint union of
  an $n$-dimensional  manifold $\inti M$ 
  and an $(n-1)$-dimensional  manifold $\dd M$ such that $\inti M$ 
  is dense in $M$  and $\dd M$ is  $(n-1)$-LCC in $M$.
  Then $M$ is an $n$-manifold
  with boundary $\dd M$.
\end{theorem}

\begin{theorem}
  \label{thm:boundary-has-spherical-cech-cohomology}
  Let $G$ be a torsion-free word-hyperbolic group.
  Let $n \geq 6$.
  \begin{enumerate}
  \item \label{thm:boundary-has-spherical-cech-cohomology:ex}
  The following statements are equivalent:
  \begin{enumerate}
    \item \label{thm:boundary-cech:cech} 
          the boundary $\dd G$ has the integral \v{C}ech cohomology of
          $S^{n-1}$;
    \item \label{thm:boundary-cech:dualtity}
          $G$ is a Poincar\'e duality group of formal dimension $n$;
    \item \label{thm:boundary-cech:homology-mfd}
          there exists a closed $\ANR$-homology manifold
          $M$ homotopy equivalent to $BG$. 
          In particular, $M$ is aspherical and $\pi_1(M) \cong G$;
  \end{enumerate}
  \item \label{thm:boundary-has-spherical-cech-cohomology:DDP} 
  If the statements in 
  assertion~\ref{thm:boundary-has-spherical-cech-cohomology:ex} 
  hold, then the homology $\ANR$-manifold $M$ appearing there
  can be arranged to have the DDP;
  \item \label{thm:boundary-has-spherical-cech-cohomology:uniqueness}
  If the statements in 
  assertion~\ref{thm:boundary-has-spherical-cech-cohomology:ex}
  hold, then the homology  $\ANR$-manifold $M$ appearing there is
  unique up to $s$-cobordism of $\ANR$-homology manifolds.
\end{enumerate}
\end{theorem}

\begin{proof}
  By~\cite[page~73]{Ghys-Harpe(1990)} torsion-free hyperbolic groups 
  admit a finite $CW$-model for $BG$.
  Thus the $\IZ G$-module $\IZ$ admits a resolution of finite
  length of finitely generated free $\IZ G$ modules.
  By~\cite[Corollary~1.3]{Bestvina-Mess(1991)} the
  $(i-1)$-th \v{C}ech cohomology of the boundary $\partial G$  
  agrees with $H^i(G;\IZ G)$.
  This shows that the statements~\ref{thm:boundary-cech:cech}
  and~\ref{thm:boundary-cech:dualtity} in 
  assertion~\ref{thm:boundary-has-spherical-cech-cohomology:ex} are 
  equivalent.

  The Farrell-Jones Conjecture in $K$- and $L$-theory holds
  by~\cite{Bartels-Lueck(2009borelhyp),Bartels-Lueck-Reich(2008hyper)}.
  This implies that the assumptions of 
  Theorem~\ref{thm:FJ_and_Borel-existence} are satisfied,
  compare~\cite[Proposition~2.2~on~page~685]{Lueck-Reich(2005)}.  
  This finishes the proof of 
  Theorem~\ref{thm:boundary-has-spherical-cech-cohomology}.
\end{proof}

\begin{proof}[Proof of Theorem~A]~\ref{the:main:sphere}
  Let $G$ be a torsion-free hyperbolic group.
  Assume that $\dd G \cong S^{n-1}$ and $n \geq 6$.
  Theorem~\ref{thm:boundary-has-spherical-cech-cohomology}
  implies that there is a closed 
  $n$-dimensional homology 
  $\ANR$-manifold $N$ homotopy equivalent to $BG$.
  Moreover, we can assume that $N$ has the \mbox{DDP}.
  The universal cover $M$ of $N$ is an $n$-dimensional
  $\ANR$-homology manifold with a proper and cocompact action 
  of $G$.
  The homotopy equivalence $N \to BG$ lifts
  to a $G$-homotopy equivalence $M \to EG$.
  For sufficiently large $d$, $P_d(G)$ is a model for $EG$ 
  (see~\cite[page~73]{Ghys-Harpe(1990)}).
  Thus there is a $G$-homotopy equivalence $M \to P_d(G)$.
  Theorem~\ref{thm:Z-set-in-X} implies that
  $\overline{M}$ is an $\ANR$ and $\dd G$ is a $Z$-set
  in $\overline{M}$.
  We conclude from Lemma~\ref{lem:overline-X} that
  $\overline{M}$ is compact and has finite dimension.
  Thus we can apply 
  Proposition~\ref{prop:Z-set-yields-homology-mfd-with-boundary}
  and deduce that $\overline{M}$ is a homology $\ANR$-manifold
  with boundary.
  Its boundary is a sphere and in particular a manifold.
  Corollary~\ref{cor:quinn-obstr_and_boundary}
  implies that $\iota(M) = 1$.
  By Theorem~\ref{thm:Quinn-obstruction}~\ref{thm:Quinn-obst:local}
  this implies $\iota(N) = 1$.
  Using
  Theorem~\ref{thm:Quinn-obstruction}~\ref{thm:Quinn-obst:manifold} 
  we deduce that
  $N$ is a topological manifold.
  By Theorem~\ref{thm:Z-set-in-X} the boundary $\dd G \cong S^{n-1}$
  is $k$-LCC in $M$ for all $k$.
  Therefore we can apply Theorem~\ref{thm:ferry-seebeck}  
  and deduce that $\overline{M}$ is a manifold with 
  boundary $S^{n-1}$.
  The $Z$-condition implies that $\overline{M}$ is 
  contractible, because $M$ is contractible as the universal
  cover of the aspherical manifold $N$.
  The $h$-cobordism theorem for topological manifolds
  implies that $\overline{M} \cong D^n$.
  In particular, $M \cong \IR^n$.
  This shows that~\ref{the:main:sphere}  
  implies~\ref{the:main:manifold}.
  The converse is obvious.
\end{proof}


\typeout{------------------- Rigidity ----------------------}

\section{Rigidity}
\label{sec:rigidity}

The uniqueness question for the manifold appearing in our result
from the introduction is a special case of the Borel Conjecture
that asserts that aspherical manifolds are topological rigid:
any isomorphism of fundamental groups of two closed 
aspherical manifolds should be realized (up to inner automorphism)
by a homeomorphism. 
The connection of this rigidity question to assembly maps is 
well-known and one of the main motivations for the Farrell-Jones Conjecture.
For homology $\ANR$-manifolds the corresponding rigidity statement
is (because of the lack of an s-cobordism theorem) somewhat weaker.

\begin{theorem}
  \label{thm:rigidity}
  Let $G$ be a torsion-free group.
  Assume that
  \begin{itemize}
  \item the (non-connective) $K$-theory assembly map
        \[
         H_i\bigl(BG;\bfK_\IZ\bigr) \to  
            K_i(\IZ G)
        \]
        is an isomorphism for $i \leq 0$ and
        surjective for $i = 1$;
  \item the (non-connective) $L$-theory assembly map 
        \[
         H_i\bigl(BG;^w\bfL^{\langle -\infty\rangle}_\IZ\bigr) \to  
            L_i^{\langle - \infty \rangle}(\IZ G,w)
        \]
        is bijective for every $i \in \IZ$ and every 
        orientation homomorphism $w \colon G \to \{\pm 1\}$.
  \end{itemize}
  Then the following holds:
  \begin{enumerate}
  \item \label{thm:rigidity:manifold}
        Let $M$ and $N$ be two aspherical closed $n$-dimensional 
        manifolds together with isomorphisms 
        $\phi_M \colon \pi_1(M) \xrightarrow{\cong} G$ and 
        $\phi_N \colon \pi_1(N) \xrightarrow{\cong} G$. 
        Suppose $n \ge 5$.
  
        Then there exists a homeomorphism $f \colon M \to N$ 
        such that $\pi_1(f)$ agrees with 
        $\phi_N \circ \phi_M^{-1}$ (up to inner automorphism);
  \item \label{the:uniqueness_ANR} 
        Let $M$ and $N$ be two aspherical closed $n$-dimensional 
        homology $\ANR$-manifolds together with
        isomorphisms $\phi_M \colon \pi_1(M) \xrightarrow{\cong} G$ 
        and $\phi_N \colon \pi_1(N) \xrightarrow{\cong} G$. 
        Suppose $n \ge 6$.
  
        Then there exists an $s$-cobordism of 
        homology $\ANR$-manifolds
        $W = (W,\partial_0 W, \partial_1 W)$, homeomorphisms
        $u_0 \colon M_0 \to \partial_0W$ and 
        $u_1 \colon M_1\to \partial_1W$ and an
        isomorphism $\phi_W \colon \pi_1(W) \to G$ such that
        $\phi_W \circ \pi_1(i_0 \circ u_0)$ and 
        $\phi_W \circ \pi_1(i_1 \circ u_1)$ agree
        (up to inner automorphism), where 
        $i_k \colon \partial_kW \to W$ is
        the inclusion for $k = 0,1$.
  \end{enumerate}
\end{theorem}

\begin{proof}~\ref{thm:rigidity:manifold}
As discussed in the proof of Theorem~\ref{thm:FJ_and_Borel-existence}
the assumptions imply that $\Wh(G) = 0$.
Therefore it suffices to show that the structure set
$\IS^{TOP}(M)$ (see \cite[Definition~18.1]{Ranicki(1992)}) 
in the Sullivan-Wall geometric surgery exact sequence 
consists of precisely one element.
This structure set is identified with the quadratic
structure group $\IS_{n+1}(M) = \IS_{n+1}(BG)$ in
\cite[Theorem~18.5]{Ranicki(1992)}.
A discussion similar to the one in the proof
of Theorem~\ref{thm:FJ_and_Borel-existence}
shows that our assumptions imply that
the quadratic structure group is trivial.
\\~\ref{the:uniqueness_ANR}
This follows from a similar argument that
uses the surgery exact sequences for homology $\ANR$-manifolds
due to Bryant-Ferry-Mio-Weinberger~%
\cite[Main Theorem on page~439]{Bryant-Ferry-Mio-Weinberger(1996)}.
\end{proof}


\section{The Quinn obstruction depends only on  the boundary}

Let $G$ be a torsion-free hyperbolic group.
Assume that $\dd G$ has the integral \v{C}ech cohomology of
a sphere $S^{n-1}$ with $n \ge 6$.
By Theorem~\ref{thm:boundary-has-spherical-cech-cohomology}
there is  a closed  aspherical $\ANR$-homology manifold $N$
whose fundamental group is $G$.

\begin{proposition}
  \label{prop:Quinn-obstruction+boundary}
  In the above situation the Quinn obstruction 
  (see Theorem~\ref{thm:Quinn-obstruction}) $\iota(N)$
  depends only on $\dd G$.
\end{proposition}

\begin{proof}
  Let $H$ be a further torsion-free hyperbolic group
  such that $\dd H \cong \dd G$.
  Let $N'$ be a closed aspherical $\ANR$-homolgy manifold
  whose fundamental group is $H$.
  Then both the universal covers $M$ of $N$ and
  $M'$ of $N'$ can be compactified to $\overline{M}$ and
  $\overline{M'}$ such that $\dd G \cong \dd H$ is a $Z$-set in both,
  see Theorem~\ref{thm:Z-set-in-X}.
  Now set $X := \overline{M} \cup_{\dd G} \overline{M'}$.
  We claim that $X$ is a connected  $\ANR$-homology manifold.
  Thus 
  \[
  \iota(N) = \iota(M) = \iota(X) = \iota(M') = \iota(N')
  \]
  by Theorem~\ref{thm:Quinn-obstruction}~\ref{thm:Quinn-obst:local}.
  To prove the claim we refer to~\cite{Ancel-Guilbault(1999)},
  see in particular pp.1270-1271.
  Both, $M$ and $M'$ are homology manifolds in the
  sense of this reference.
  By fact~6 of this reference, 
  $X$ is also a homology manifold.
  It remains to show that $X$ is an $\ANR$.
  This follows from an argument given
  during the proof of Theorem~9 of this reference.
\end{proof}

\begin{proof}[Proof of Theorem~B]
  Let $G$ and $H$ be torsion-free hyperbolic groups, such that
  $\dd G \cong \dd H$.
  Assume that $G$ is the fundamental group of a closed aspherical
  manifold of dimension at least $6$.
  Theorem~\ref{thm:boundary-has-spherical-cech-cohomology}~%
  \ref{thm:boundary-has-spherical-cech-cohomology:ex}
  implies that $\dd G \cong \dd H$ has the integral \v{C}ech cohomology of
  a sphere $S^{n-1}$ with $n \ge 6$
  and that $H$ is the fundamental group of a closed aspherical
  $\ANR$-homology manifold $M$ of dimension $n$.
  Because of Theorem~\ref{thm:boundary-has-spherical-cech-cohomology}%
  ~\ref{thm:boundary-has-spherical-cech-cohomology:DDP} 
  this $\ANR$-homology manifold can be arranged to have the DDP.
  Now by Proposition~\ref{prop:Quinn-obstruction+boundary}
  (and Theorem~\ref{thm:Quinn-obstruction}~\ref{thm:Quinn-obst:manifold})
  we have $\iota(M) = 1$.
  Using Theorem~\ref{thm:Quinn-obstruction}~\ref{thm:Quinn-obst:manifold}
  again, it follows that $M$ is a manifold.

  A similar argument works if $G$ is the fundamental group of
  closed aspherical homology $\ANR$-manifold that is not necessary
  a closed manifold.
\end{proof}


\section{Exotic examples}

In light of the results of this paper 
one might be tempted to wonder if for a 
torsion-free hyperbolic group $G$, the condition $\dd G \cong S^n$
is equivalent to the existence of a closed aspherical
manifold whose fundamental group is $G$.
This is however not correct: 
Davis-Januszkiewicz and Charney-Davis
constructed closed aspherical manifolds whose
fundamental group is hyperbolic with boundary not homeomorphic to a sphere.
We review these examples below.

\begin{example}
  \begin{enumerate}
  \item For every $n \ge 5$ there exists an example of an aspherical closed 
    topological  manifold $M$ of dimension $n$  
    which is a piecewise flat, non-positively curved polyhedron 
    such that the universal covering $\widetilde{M}$ is not 
    homeomorphic to Euclidean space
    (see~\cite[Theorem~5b.1 on page~383]{Davis-Januszkiewicz(1991)}).
    There is a variation of this construction that uses the 
    strict hyperbolization of Charney-Davis~\cite{Charney-Davis(1995)} 
    and produces closed aspherical manifolds whose universal cover is 
    not homeomorphic to 
    Euclidean space and whose fundamental group is hyperbolic.
  \item For every $n \ge 5$ there exists a strictly negative 
    curved polyhedron of dimension $n$
    whose fundamental group $G$ is hyperbolic, which is homeomorphic to a
    closed aspherical smooth manifold and 
    whose universal covering is homeomorphic to $\IR^n$, but 
    the boundary $\partial G$ is not homeomorphic to $S^{n-1}$, 
    see~\cite[Theorem~5c.1 on page~384 and Remark on page~386]
    {Davis-Januszkiewicz(1991)}.
  \end{enumerate}
\end{example}

On the other hand, one might wonder if assertion~\ref{the:main:manifold}
in Theorem~A can be strengthed to the existence of more structure on the
aspherical manifold. 
Strict hyperbolization~\cite{Charney-Davis(1995)} can be used 
to show that in general there may be no smooth closed aspherical
manifold in this situation.

\begin{example}
  Let $M$ be a closed oriented triangulated $\PL$-manifold.
  It follows from~\cite[Theorem~7.6]{Charney-Davis(1995)} that  there is a 
  hyperbolization $\calh(M)$ of $M$ has the following properties:
  \begin{enumerate}
  \item $\calh(M)$ is a closed oriented $\PL$-manifold.
      (This uses properties (2) and (4) 
         from~\cite[p.333]{Charney-Davis(1995)}.) 
  \item There is a degree $1$-map $\calh(M) \to M$ under
      which the rational Pontrjagin classes of $M$ pull back
      to those of $\calh(M)$. 
      In particular, the Pontrjagin numbers of $M$ and $\calh(M)$
      conincide.
      (See properties (5) and (6)' from~\cite[p.333]{Charney-Davis(1995)}.)
  \item $M$ is a negatively curved piece-wise hyperbolic polyhedra. 
      In particular  $G := \pi_1(\calh(M))$ is hyperbolic.
      Moreover, by \cite[p.~348]{Davis-Januszkiewicz(1991)} the
      boundary of $\dd G$ is a sphere. 
  \end{enumerate}
  Suppose that some Pontrjagin number of $M$ is not an integer.
  Then the same is true for $\calh(M)$.
  In particular $\calh(M)$ does not carry the structure of a smooth manifold.
  If in addition $\dim \calh(M) = \dim M \geq 5$, 
  then by Theorem~\ref{thm:rigidity}~\ref{thm:rigidity:manifold}
  any other closed aspherical manifold $N$ with $\pi_1(N) = G$
  is homeomorphic to $M$ and does not carry a smooth structrue either.
  Such manifolds $M$ exist in all dimensions $4k$, $k \geq 2$,
  see Lemma~\ref{lem:non-integer-pontrjagin}.
  This shows that there are for all $k \geq 2$ torsion-free hyperbolic
  groups $G$ with $\dd G \cong S^{4k-1}$ that are not fundamental groups of
  smooth closed aspherical manifolds. 
  In particular such a $G$ is not the fundamental group of a 
  Riemannian manifolds of non-positive curvature.
\end{example}

In the previuous example we needed $PL$-manifolds that do not carry
a smooth structure. 
Such manifolds are classically contructed using 
Hirzebruch's Signature Theorem.

\begin{lemma}
  \label{lem:non-integer-pontrjagin}
  Let $k \geq 2$. There is an oriented closed $4k$ dimensional $PL$-manifold
  $M^{4k}$ whose top Pontrjagin number 
  $\langle p_k(M^{4k}) \mid [M^{4k}] \rangle$ is not an integer. 
\end{lemma}

\begin{proof}
  For all $k \geq 2$ there are smooth framed compact manifolds $N^{4k}$
  whose signature is $8$ and whose boundary is a $4k-1$-homotopy 
  sphere, see~\cite{Browder(1972)} and~\cite[Theorem~3.4]{Levine(1983)}.
  By~\cite{Smale(1961)} this homotopy sphere is $PL$-isomorphic
  to a sphere.
  We can now cone off the boundary and obtain a $PL$-manifold $M^{4k}$
  (often called the Milnor manifold)
  whose only nontrivial Pontrjagin class is $p_k$ and 
  whose signature $\sigma(M^{2k})$ is $8$.
  Hirzebruch's Signature Theorem implies that
  \begin{equation*}
    8 = \sigma(M^{4k}) = \frac{2^{2k}(2^{2k-1}-1)B_k}{2k!} 
         \langle p_k(M^{4k}) \mid [M^{4k}] \rangle
  \end{equation*}
  where $B_k$ is the $k$-th Bernoulli number, 
  see~\cite[p.~75]{Levine(1983)}.
  For $k = 2,3$ we have then
  \begin{equation*}
     8 = \frac{7}{45} \langle p_2(M^{8}) \mid [M^8] \rangle
       = \frac{62}{945} \langle p_3(M^{12}) \mid [M^{12}] \rangle
  \end{equation*}
  compare~\cite[p.225]{Milnor-Stasheff(1974)}.
  This yields examples for $k = 2,3$.
  Taking products of these examples we obtain examples for all $k \geq 2$.
\end{proof}


\typeout{-------- Open questions}

\section{Open questions}

We conclude this paper with two open questions.

\begin{enumerate}
\item Can the boundary of a hyperbolic group
      be a  $\ANR$-homology sphere that is not a sphere? 
\item Can one give an example of a hyperbolic group 
      (with torsion)
      whose boundary is a sphere, 
      such that the group does not act properly
      discontinuously on some contractible manifold?
\end{enumerate}


\typeout{-------------------- References ---------------------}

\end{document}